\documentclass[a4paper, 12pt]{article}
\usepackage[utf8]{inputenc}
\usepackage[margin=1.1in]{geometry} 
\usepackage[T1]{fontenc}
\usepackage{tgtermes}
\usepackage{amsthm}
\usepackage{amsmath}
\usepackage{bm}
\usepackage{enumitem}
\usepackage{amsfonts}
\usepackage{amssymb}
\usepackage{mathtools}
\usepackage{appendix}
\usepackage{mathrsfs}
\usepackage{setspace}
\usepackage{comment}
\usepackage{xcolor}
\usepackage{thmtools} 
\usepackage[pdfdisplaydoctitle,colorlinks,breaklinks,urlcolor=blue,linkcolor=blue,citecolor=blue]{hyperref} 
\usepackage[nameinlink,capitalise]{cleveref}

\newcommand{\B}{\mathcal{B}}

\newcommand{\D}{\mathcal{D}}

\newenvironment{acknowledgements}{%
  \begin{abstract}
}{%
  \end{abstract}
}

\newcommand{\EE}{\mathbb{E}}
\newcommand{\F}{\mathcal{F}}
\newcommand{\G}{\mathcal{G}}

\newcommand{\X}{\mathbb{X}}

\renewcommand{\L}{\mathscr{L}}

\newcommand{\N}{\mathbb{N}}
\newcommand{\NN}{\mathcal{N}}

\newcommand{\PP}{\mathbb{P}}

\newcommand{\R}{\mathbb{R}}

\DeclareMathOperator{\trace}{Tr}

\DeclareMathOperator{\weakstarto}{\overset{\ast}{\rightharpoonup}}

\renewcommand{\epsilon}{\varepsilon}

\newcommand{\one}{\bm{1}}

\newcommand{\bra}[1]{\left[#1\right]}

\newcommand{\brak}[1]{\left\langle#1\right\rangle}

\newtheorem{theorem}{Theorem}[section]
\newtheorem{definition}[theorem]{Definition}
\newtheorem{assumption}[theorem]{Assumption}

\newtheorem{corollary}[theorem]{Corollary}
\newtheorem{lemma}[theorem]{Lemma}
\newtheorem{proposition}[theorem]{Proposition}

\theoremstyle{remark}
\newtheorem{remark}[theorem]{Remark}
\newtheorem{ex}[theorem]{Example}

\numberwithin{equation}{section}

\makeatletter
\newcommand{\subjclass}[2][2020]{%
  \let\@oldtitle\@title%
  \gdef\@title{\@oldtitle\footnotetext{#1 \emph{Mathematics subject classification.} #2}}%
}
\newcommand{\keywords}[1]{%
  \let\@@oldtitle\@title%
  \gdef\@title{\@@oldtitle\footnotetext{\emph{Key words and phrases.} #1.}}%
}
\makeatother
\title{Ergodicity for a class of stochastic Hasegawa-Mima equations with dissipation}
\author{Federico Butori \thanks{Scuola Normale Superiore, Piazza dei Cavalieri 7, Pisa, Italy.
  \texttt{federico.butori@sns.it}} \and Ciro Campolina \thanks{Scuola Normale Superiore, Piazza dei Cavalieri 7, Pisa, Italy.
  \texttt{ciro.campolina@sns.it}}\and Franco Flandoli \thanks{Scuola Normale Superiore, Piazza dei Cavalieri 7, Pisa, Italy.
  \texttt{franco.flandoli@sns.it}}}
\date{}
\allowdisplaybreaks

\begin{document}
\maketitle
\begin{abstract}
	This work deals with the ergodicity of the stochastic dissipative Hasegawa--Mima equations in a bounded domain driven by additive noise. The uniqueness of the invariant measure is obtained by the asymptotic coupling technique. 
\end{abstract}

\section{Introduction}
We are interested in establishing the uniqueness of invariant measures for the following stochastic Hasegawa-Mima equation with dissipation:
\begin{equation}\label{eq:main}
	d(\phi -\Delta \phi) + \left(A \phi + \nabla^\perp \phi\cdot \nabla \left(v-\Delta \phi  \right)\right)dt =  dW_t\\
\end{equation}
where $v= \log n_0$ and $n_0$ is a smooth background density field, $\D$ is a smooth domain of $\R^2$, $W$ is a white noise appropriately colored in space, and $A$ is some operator modelling dissipation. To close the system we have to specify boundary conditions. We will always assume $\phi_{|\partial \D}=0$, which would be sufficient to close the system in the ideal case ($A=0$). In the dissipative case of interest here,
due to a lack of general consensus in the literature on the choice of a dissipative mechanism and, in turn, on boundary conditions, we remain flexible providing a functional framework based on abstract assumptions on the operator $A$. We provide two concrete examples of dissipative operators compatible with our framework in \Cref{example_op}. 
The precise interpretation of the equation is then given in Section~\ref{sec:main} below and
the physical meaning is described in Subsection \ref{Subsect physics}. 

This equation may be rewritten in a form similar to the stochastic
Navier-Stokes equations. Namely, we can interpret $\phi$ as a stream function so that $\nabla^\perp \phi$ represents a divergence-free velocity field while $\Delta \phi$ plays the role of the vorticity field. Thus, the nonlinearity contains an active transport term analogous to the one in the vorticity formulation of 2D Navier-Stokes. Thanks to this similarity, the solution theory of the deterministic equation shares some similarities with the Navier-Stokes equations, and indeed in the ideal case $(A=0)$, Yudovich-type solutions can be constructed as recently proven by Flandoli and Tahraoui in \cite{flandoli_hasegawamima_2025a}. 
The purpose of this paper is to prove the ergodicity of a stochastic and dissipative version of these equations.

The first result of
ergodicity for the stochastic Navier-Stokes equations is due to Flandoli and
Maslowski \cite{flandoli_ergodicity_1995}; the approach followed in that paper applies a theorem of Doob, properly revised in infinite dimensions,
following \cite{daprato_stochastic_}; a feature of that method is the need for a noise
acting on all Fourier components. The paper \cite{flandoli_ergodicity_1995} originated a
long stream of research, mostly aimed at reducing the number of Fourier
components needed for the ergodicity --- notice that existence of an invariant
measure holds even without noise --- with contributions from many authors. To
mention only a few, there were intermediate results under different
assumptions by \cite{KuksinS}, \cite{E}, \cite{BrichKupiainen} and many
others, and some outstanding contributions, either for their generality or for
the applicability to many other systems, namely \cite{glatt-holtz_unique_2017},
\cite{hairer_ergodicity_2006} (again among others with important variants, like
\cite{Bukow}). Ergodicity was of course considered for many other models,
different from the 2D Navier-Stokes equations; it is impossible to review the
contributions. The present paper is the application of the methodology of
\cite{glatt-holtz_unique_2017} to the model \eqref{eq:main}. At the same time, the estimates required to apply such techniques share a lot of analogies with the proof of continuous time data assimilation. See \cite{bessaih2025continuous, bessaih2015continuous} for examples of results and in particular \cite{farhat2020data, oljaca2018almost} for abstract results that seem to cover the model introduced here.

\subsection{About the physics and open questions\label{Subsect physics}}

Equation \eqref{eq:main} is a stochastic version of the Hasegawa-Mima
equation (HME) \cite{HasegawaMima}, \cite{Horton}. The HME has been introduced
around 1977 as a reduced model of a plasma observed under certain space-time
scales - another famous model under different scales, just to compare, are the
reduced magnetohydrodynamic equations. In spite of its simplicity, the HME
should incorporate sufficient phenomena to observe features like turbulence,
waves, and stretching by density gradient, known to exist in plasma. Compared
to the classical 2D Euler or Navier-Stokes equations, here there are in addition waves and particle density gradients,
typical of plasma.

The physical meaning of the solution $\phi\left(  x,t\right)  $ is of an
electrostatic potential and the vector field $u\left(  x,t\right)
:=\nabla^{\perp}\phi\left(  x,t\right)  $ has the meaning of ion velocity ---
more precisely, it is a first-order approximation of the true velocity, which
should contain also some compressible lower-order terms. In the analogy with
the Navier-Stokes equations, $\phi\left(  x,t\right)  $ corresponds to the
stream function, but there we only have the term $\partial_{t}\Delta\phi$
(where $\Delta\phi$ is the vorticity) instead of $\partial_{t}\left(
\phi-\Delta\phi\right)  $. The term $n_{0}\left(  x\right)$, appearing in the definition of $v$, is the
(background) average particle density. 

The classical HME \cite{HasegawaMima} is inviscid and not forced by noise.
Viscosity or other forms of damping/diffusion have been introduced many times
in later literature, but since there is no general agreement on the precise
form, in our model below we have made a generic choice of an unbounded
operator with certain properties which covers various cases. 

Concerning the noise, that, as mentioned above, was not included in the original
model, it is motivated by the injection of energy at small scales (comparable
to the so-called Larmour radius scale) by instabilities, like the
ITG\ instability. The stochastic model has already been introduced in the
physics literature, see for instance \cite{Krommes}, \cite{Kim}. When we
introduce an additive noise, we have to balance it by a dissipation, to keep
an average energy in the system, hence the need of a dissipation term.

Among the questions about the HME as a plasma model that are partially open
and require more investigation, we may quote the properties of the inverse
cascade - for this reason it is relevant to force it at small scales - and the
properties of the large-scale structures emerging from the cascade, the
interaction with the drift wave dynamics and the density-stretching near the
boundary. All these ingredients are incorporated a priori in the HME but it is
not easy to extract information from mathematical studies. The invariant
measure (introduced in this paper) is conceptually the basic object which
should contain part of such information. Since even for the Navier-Stokes
equations the understanding of statistical and turbulence properties from the
invariant measure is a difficult problem, the same is true for the HME. But it
is important to settle the equation and the ergodicity of the invariant
measure in rigorous terms in order to start the investigation.

Concerning the noise, an interesting and very difficult mathematical question
arises. The most natural noise, as said above, is one acting on certain small
scales, to describe small-scale instabilities which inject energy into the
plasma and perturb a potentially steady state. However, from the purely
mathematical viewpoint, if the noise acts only on a few small-scale
components, it is difficult to understand the global properties of the system.
Heuristically, the energy injected at
small scales "cascades" to larger scales, a typical phenomenon of
two-dimensional fluid systems; and simultaneously some flux of information to
smaller scales also exists - the so-called direct enstrophy cascade. The
"signal" propagates among Fourier modes and spreads in all possible
directions. However, proving and controlling this process is very difficult.
In the present paper we limit ourselves to activating all Fourier components up
to a certain level, but we cannot restrict ourselves to activating only
\textit{some} small-scale components; the technique developed by \cite{glatt-holtz_unique_2017}
that we employ does not allow it. Conceptually, the approach of
\cite{hairer_ergodicity_2006} goes in this direction but much work has to be done to
make it work in the present case. Notice, for comparison, that for the
classical stochastic Navier-Stokes equations the main literature claimed that
the interesting model is noise only at the largest scales, cf.
\cite{BrichKupiainen}, which motivated \cite{glatt-holtz_unique_2017} and \cite{hairer_ergodicity_2006}
and many others. In the case of the HME, the natural activation is at certain
small-scales and this should be better understood from the viewpoint of
ergodicity. Let us mention, in this direction, \cite{RomitoXu} where all modes
from a certain scale to the smallest ones (infinitely many, as in Doob theory)
are assumed to be activated by the noise.

\section{Preliminaries}
	\subsection*{Notation}
	In the following, we will employ the following notation: given a smooth domain $\D\in \R^2$, $\eta$ represents the unit outward normal to $\partial \D$; given a vector $v=(v_1, v_2)\in \R^2$, $v^\perp = (v_2, -v_1)$ and for a scalar function $f$, $\nabla^\perp f := (\nabla f)^\perp $; $\cdot$ will represent the scalar product in $\R^d$; given two vectors in $\R^d$, we denote their Kronecker product as $(v\otimes w)_{ij} = v_iv_j$; we denote with the apex $^t$ the matrix transpose and for any two matrices $A, B$ we denote their scalar product with $A:B= \trace(AB^t)$.
\subsection{Functional Setting}\label{sec:funct-sett}
Let $\mathcal{D \subseteq\R}^2$ be a bounded smooth domain. We denote by $C^{\infty}_c(\mathcal{D})$ the set of smooth functions on $\D$ with compact support. 
We will use $|\cdot |$ for the norm on $L^2(\D)$ and $\brak{\cdot, \cdot}$ to denote the usual $L^2$ inner product while we use $|\cdot|_{p}$ for $L^p(\D)$ norms for $p \neq 2$. With some abuse of notation, we use the same symbol $\brak{\cdot, \cdot}$ also for vector and matrix-valued functions, implicitly replacing the product with the scalar product $\cdot$ and with the product $:$, respectively. 
We denote by $W^{s,p}(\D)$ the usual Sobolev spaces, and we denote with $W^{s,p}_0$ the subspace consisting of the closure of $C^{\infty}_c(\D)$. Given $s\in (0,1)$, $p\ge 1$ and a separable Banach space $\X$ we will consider also the vector-valued fractional Sobolev spaces $W^{s,p}([0, T]; \X)$ endowed with the norm
$$\|f\|_{W^{s,p}} = |f|_{p } +  \left(\int_0^T \int _0^T \frac{\|f(t) - f(s)\|_\X^p}{|t-s|^{1+sp}}dsdt\right)^{1/p}.$$
We consider on the space $L^2(\D)$ the self-adjoint and positive operator $-\Delta :W^{2,2}(\D) \cap W^{1,2}_0(\D)\subset L^2(\D) \rightarrow L^2(\D)$. We set $V:= W^{1,2}_0(\D) = D((-\Delta)^{1/2})$ endowed with the graph norm $\|x\|_V= |(I-\Delta)^{1/2}x| $. In particular we consider on $V$ the scalar product
$$\brak{x,y}_V:= \brak{(I-\Delta)^{1/2}x, (I-\Delta)^{1/2}y} = \brak{x,y} + \brak{\nabla x, \nabla y}.$$
The reason for considering the non-homogeneous norm on $V$ comes from the fact that since the dynamical variable in \eqref{eq:main} is $\phi-\Delta\phi$, this is the natural energy space for our equation. 
Finally we set $V'=V^{-1}$ the dual space of $V$ in the sense that the duality extends the $L^2$ inner product. We recall also the Poincaré inequality $|x|_p \lesssim_p |\nabla x|_p$, for every $p\ge 1$ and $x \in W_0^{1,p}$.  
For $s\in \R$ we set $V^s:= D((-\Delta)^{s/2})$. We recall that for any $s>r$, $V^s$ embeds compactly in $V^r$.

Finally, we make the following assumptions on $A$:

\begin{assumption} \label{ass}
	The operator $A$ satisfies:
	\begin{enumerate} 
		\item $A: D(A) \subset L^2(\D) \rightarrow L^2(\D)$ is densely defined, self-adjoint and strictly positive.
		\item The operator $A$ is coercive, i.e. there exists $C>0$ such that  $\brak{Ax, x} \ge C |x|^2$, for all $x\in D(A)$. Thus we can set $$W := D(A^{1/2}), \qquad  \| x \|_W := | A^{1/2} x |.$$ 
		\item\label{ass1.2}   $D(A)\subset D(I-\Delta)$ and we have the continuous embedding $W \hookrightarrow V^2$ (i.e. $A$ is also $V^2-$coercive). In particular $W$ embeds compactly in $V$.
	\end{enumerate}
\end{assumption}

\begin{remark}[\emph{A special basis}]\label{rem:onb}
	Due to  \Cref{ass}, and in particular the compactness of the embedding $W\hookrightarrow V$, thanks to (cf. \cite[Theorem 4.1]{laugesen_spectral_2012}) we can construct an orthonormal basis $\{e_k\}_{k\ge 1}$ of $V$ composed by `eigenvectors' of $A$ with respect to $\brak{\cdot, \cdot }_V$ and corresponding eigenvalues $0 < \lambda_1 \le \lambda_2 \le \ldots \rightarrow + \infty$, i.e. such that for every $v \in W$ it holds
	\begin{equation}\label{eq:eigen-def}\brak{A^{1/2}e_k, A^{1/2}v}= \lambda_k \brak{e_k, v}_V.\end{equation}
	In particular, the eigenvectors $e_k$ are orthogonal both in $W$ and in $V$.
	Moreover, all $e_k \in D(A)$.
	Indeed from $D(A^{1/2})\subset V^2$ we have from \eqref{eq:eigen-def}, for every $v\in D(A)$, $\brak{e_k, Av}\le \lambda_k \|e_k\|_{V^2}|v|\lesssim  \|e_k\|_{W}|v|$, so that $e_k \in D(A^*)=D(A)$.
\end{remark}
Consider $P_N:V \rightarrow V_N$ the orthogonal projection of $V$ onto $V_N=\operatorname{span}\{e_1, \ldots e_N\}$, defined as $P_Nx = \sum_{k=1}^N \brak{x, e_k}_V e_k$ and $ Q_N= I- P_N$. Because $e_k \in D(A)$, we have that $P_Nx = \sum_{k=1}^N\brak{x, (I-\Delta)e_k}e_k$ naturally extends to the whole $L^2$. In what follows we will make use of the lemma below. 

\begin{lemma}\label{lem:proj-est2}
	There exists $C>0$ such that for all $x\in D(A^{1/2})$ we have 
	\begin{align}
		\|x\|_{W}^2 &\ge  \lambda_N\|Q_N x\|_V^2 \\
		\|x\|^2_W &+ \lambda _N \|P_Nx\|_V^2 \ge \lambda_N\|x\|_V^2
	\end{align}
\end{lemma}
\begin{proof}
	By the properties of the basis we have $\|x \|_W^2 = |A^{1/2}x|^2 = \sum_{k=1}^{\infty} \lambda_k\brak{x, e_k}_V^2$. On the other hand $\|P_Nx\|_V^2 = \sum_{k=1}^N \brak{x, e_k}_V^2$, and we easily obtain the desired inequalities.
\end{proof}
We will also denote with $W'$ the dual of $W$ where the duality is that which extends the $L^2$ inner product. 
We set in general $W^\beta=D(A^{\beta /2})$ endowed with its natural norm.\\
Since  $D(A)\subset D(I-\Delta)$, then it is well defined the operator $B:=A(I-\Delta)^{-1} : D(B)\subset L^2 \rightarrow L^2$, with $D(B)=(I-\Delta)D(A)$. We formulate the following additional assumptions on $B$:
\begin{assumption} \label{ass2}
	The operator $B$ satisfies:
	\begin{enumerate} 
		\item The operator $B$ is densely defined positive and self adjoint. Thus it is well defined $Y:=D(B^{1/2})$ endowed with the norm $\|x\|_Y := |B^{1/2}x|$.
		\item  We have the continuous embedding $Y \hookrightarrow V$. In particular $Y$ embeds compactly in $L^2$
	\end{enumerate}
\end{assumption}
We now provide examples of operators fulfilling \Cref{ass} and clarifying the role of the additional assumption \Cref{ass2}. 
\begin{ex}\label{example_op}
	For any $\beta\ge0$, the operators $A_n, A_0$ given by 
	\begin{align*}
		D(A_n)&=  \left\{ \phi \in W^{4,2} : \phi_{|\partial D}=0 , \ \nabla \phi \cdot \eta _{|\partial D}=0\right\}\\
		D(A_0)&=\left\{ \phi \in W^{4,2} : \phi_{|\partial D}=0 , \ \Delta \phi _{|\partial D}=0\right\}
	\end{align*} 
	defined by 
	$$A_i\phi = -\Delta( \beta\phi - \Delta \phi), \quad  i=\{n, 0\}$$
	both satisfy \Cref{ass}.
	Moreover we have 
	$$
	V_0= D(A_0^{1/2}) = W^{2,2} \cap V, 
	$$
	$$ V_n= D(A_n ^{1/2})= W^{2,2}_0=  \left\{ \phi \in W^{2,2} : \phi_{|\partial D}=0 , \ \nabla \phi \cdot \eta _{|\partial D}=0\right\} $$
	However, only $A_0$ satisfies also \Cref{ass2}, indeed $D(A_0)= D((-\Delta)^2)$ and it is easy to see by integration by parts that it is self-adjoint and positive. On the other hand, due to the boundary conditions, $A_n$ does not commute with $(I-\Delta)$, and $B$ is not self adjoint. \\
	As an analogy with the Navier-Stokes equations, the operator $A_n$ corresponds to the usual `no-slip' boundary conditions, while $A_0$ would correspond to the (unphysical) Dirichlet boundary conditions on the vorticity. In the case of a fusion plasma in a tokamak modelled by the Hasegawa-Mima equations, since at present there is no general consensus on what its physical boundary should be, it is reasonable to consider different boundary conditions, hence the reason to provide a flexible functional setting to accommodate for both, and more. 
\end{ex}
We end this paragraph by recalling some properties of the nonlinearity appearing in the equation. Define $$\mathcal{N}(\phi, \psi) := \nabla^\perp \phi \cdot \nabla  \psi.$$
It holds
\begin{lemma}\label{N_properties}
	For all $\phi\in C^\infty(\D)\cap V,$ and $\psi, z \in C^\infty(\D)$,  $\mathcal{N}$ is anti-symmetric i.e. $\mathcal{N}(\phi, \psi) = - \mathcal{N}(\psi, \phi)$, moreover it holds 
	\begin{align}
		&\int_\D \mathcal{N}(\phi, \psi) z = - \int \mathcal{N}(\phi, z) \psi,    \label{N-properties 1}\\
		& \int_\D \mathcal{N}(\phi, \psi) \phi  =0, \label{N-properties 3} \\
		&\int_\D \mathcal{N}(\phi, \Delta \phi) z = \int_D \nabla^\perp \phi \otimes \nabla^\perp \phi : \nabla \nabla^\perp z. \label{N-properties 2}
	\end{align}
	By density, the bilinear operator $\mathcal{N}$ can be extended as an operator $\mathcal{N}:(W^{1,4}\cap V)\times L^4 \rightarrow V'$ so that \eqref{N-properties 3} holds for all $\psi\in L^4, \phi \in W^{1,4}$. 
	Define also the operator $\mathcal{N}(\phi):= \mathcal{N}(\phi, \Delta \phi)$. Then continuously $\mathcal{N}: V^2\rightarrow V^{-2}$. 
\end{lemma}
\begin{proof}
	The first identity follows by a simple integration by parts and it implies the second one. 
	The third identity follows by noticing that $\mathcal{N}(\phi, \Delta \phi)= \nabla^\perp \cdot \operatorname{div}(\nabla^\perp \phi \otimes \nabla^\perp \phi)$ and a double integration by parts. We remark that these identities holds only under the assumption $\nabla^\perp \phi\cdot \eta_{|\partial \D}=0$ (which follows from $\phi \in C^\infty(\D)\cap V$), without assuming any boundary conditions on $\psi,\ z$.
	Using H\"older inequality we get from \eqref{N-properties 1}
	$$\left|\int_\D\mathcal{N}(\phi, \psi) z \right|\le |\psi|_{L^4}|\nabla \phi|_{L^4}|\nabla z| . $$
	From which the extension property holds by density. Finally, thanks to \eqref{N-properties 2}, we have by H\"older and interpolation 
	\begin{equation}\label{N_estimate_trace}
		\left|\int_\D\mathcal{N}(\phi, \Delta \phi) z \right|\le |\nabla \phi|_{L^4}^2|D^2 z| \lesssim  |\nabla \phi| |D^2 \phi||D^2 z| , 
	\end{equation}
	where $D^2\phi$ represents the Hessian of $\phi$. By density $\mathcal{N}(\phi)$ can be extended to $V^2$ as an operator taking values in $V'$. The continuity follows from the fact that $\phi^n \rightarrow \phi$ in $V^2$ implies $\nabla^\perp\phi^n \rightarrow \nabla^\perp \phi$ in $L^4$ by Sobolev embedding, so that $\nabla^\perp \phi^n \otimes \nabla^\perp \phi^n$ converges in $L^2$ to $\nabla^\perp \phi \otimes \nabla^\perp \phi$ by bilinearity.  
\end{proof}
\subsection{Probabilistic Setting}
Let $(\Omega, \F, \F_t, \PP)$ be a filtered probability space  on which it is defined an infinite sequence of i.i.d Brownian motions $(W_t^k)_{k\ge 1}$. 
When $\X$ is a separable Hilbert space we denote with $L^p_{\F_t}(\Omega, \X)$ the space of $\F_t$-measurable $p$-integrable $\X$-valued random variables. Moreover we denote by $C_\F(0, T; \X)$ the space of continuous adapted $\X$-valued stochastic processes $(X_t)_{t\in [0,T]}$ such that 
$$\EE\left[ \sup_{t\in [0, T]} \|X_t\|^2_\X\right] \le \infty$$
and by $L^2_\F(0, T; \X)$ the space of square integrable processes such that 
$$\EE\left[ \int_0^T \|X_s\|_\X^2 ds\right] \le \infty$$
Given a sequence of non-zero real numbers $(\sigma_k)_{k\ge 1}$ and a sequence of functions $\{f_k \}_{k\ge 1}\subset V'$, we define for some $M\ge 1$ (possibly also infinite): 
\begin{equation}
	W(t,x):= \sum_{k=1}^M \sigma_k f_k(x)W_t^k.
\end{equation}
In the case that $M=+\infty$ we require $\kappa:=\sum_{k=1}^\infty \sigma_k^2|(I-\Delta)^{-1/2} f_k|^2 <\infty$. Here $\kappa$ represents the mean instantaneous energy injected by the noise into the system.
\begin{remark}\label{rem:noise_choice}
	With the previous assumptions, $W(t)$ is a $V'$-valued Wiener process. A natural choice for $f_k$ would be $f_k = (I-\Delta)e_k$ where $\{e_k\}$ is the ONB of $V$ constructed in \Cref{rem:onb}. In this case $\kappa = \sum_{k=1}^M \sigma_k^2$. 
\end{remark}
\subsection{Invariant Measures and Ergodicity}\label{sec:invariant}
In this section we recall the main definitions regarding invariant measures for Markov Processes. For a formal treatment we refer to \cite[Chapter 9]{daprato_stochastic_}. Let $\X$ be a Polish space. We will denote with $\B(\X)$ the Borel sets of $\X$ and with $\mathcal{P}(\X)$ the set of probability measures on $(\X, \B(\X))$.
Let $(\Omega, \F, \F_t, \PP)$ be a filtered probability space on which it is defined a family of Markov processes $x(t; x_0)$ with values in $\X$ such that $x(0;x_0)=x_0$ for any $x_0 \in \X$.
Let $B(\X)$ denote the set of bounded measurable real-valued functions on $\X$
and denote by $C_b(\X)$ the subset of $B(\X)$ consisting of continuous functions. We call a \emph{Markov transition kernel} the operator $P_t: B(\X)\rightarrow B(\X)$ defined by
$$P_t\Psi(x_0):= \EE\left[ \Psi(x(t; x_0))\right], \qquad  \Psi \in B(\X).$$
Thanks to the Markov property of the process $x$, it holds $P_t\circ P_s \Psi = P_{t+s} \Psi$ , thus $P_t$ is a well defined semigroup of operators. Taking $\Psi= \one _{\Gamma}$ for some $\Gamma \in \B(\X)$ we set $P_t (x,\Gamma):= P_t\one_\Gamma(x)$. Clearly $P_t(x, \cdot)$ is a probability measure on $\X$. $P_t$ has a natural family of adjoint operators $P_t^*$ acting on $\mathcal{P}(\X)$ defined by 
\begin{equation*}
	P_t^* \mu (\Gamma) = \int_{\X} P_t(x, \Gamma)\mu(dx), \qquad \Gamma \in \B(\X). 
\end{equation*}
\begin{definition}
	The transition semigroup $P_t, \ t\ge 0$ is said to posses the Feller property if it has the mapping property $P: C_b(\X) \rightarrow C_b(\X)$, $t\ge 0$. 
\end{definition}
\begin{definition}
	$\mu\in \mathcal{P}(\X)$ is called an invariant measure for the process $x_t$ if  
	$$\int_\X P_t \Psi d\mu = \int_\X \Psi d\mu \qquad \forall \Psi\in C_b(\X), \quad  t\ge 0.$$
	Namely $P_t^* \mu = \mu$.
\end{definition}
A classical method of constructing invariant measures is the Krylov-Bogoliubov method which we recall from \cite[Theorem 11.7, Corollary 11.8]{daprato_stochastic_}.
Let us introduce the time-averaged measure
\begin{equation*}
	R_T(x, \cdot ):= \frac{1}{T}\int_0^T P_t(x, \cdot )dt.
\end{equation*}
For every $\nu \in \mathcal{P}(\X)$, $R_T^*\nu$ is defined in the natural way as 
\begin{equation*}
	R_T^*\nu(\Gamma)= \int_{\X} R_T(x,\Gamma)d\nu(x),
\end{equation*}
and  for every $\Psi \in B(\X)$,
\begin{equation*}
	\brak{  R_T^*\nu, \Psi}= \frac{1}{T}\int_0^T \brak{P_t^* \nu, \Psi} dt.
\end{equation*}
\begin{theorem}\label{thm:KB}
	Let $P_t$ be a Feller transition semigroup. If for some $\nu \in \mathcal{P}(\X)$ and some sequence $T_n \uparrow \infty$, the sequence $\{R_{T_n}^*\nu\}_{n\ge 1}$ is tight, then there exists an invariant measure for $P_t, \ t\ge 0$. 
\end{theorem}
Finally, we recall the notion of ergodicity. 
	\begin{definition}
		Let $\mu$ be an invariant measure for $P_t$. We say that $\mu $ is ergodic if
		$$\lim_{T\rightarrow\infty}\frac{1}{T}\int_0^T P_t\Psi dt = \int_\X \Psi(\phi)\mu(d\phi) \qquad \forall \Psi \in L^2(\X; \mu).$$
	\end{definition}
	For equivalent characterizations of ergodicity see \cite[Theorem 11.9]{daprato_stochastic_}
\subsubsection{The Asymptotic Coupling Method}\label{sec:coupling-method}
In this section we recall the main ingredient of the technique developed in \cite{glatt-holtz_unique_2017}.
Let $\X$ be a Polish space with metric $\rho$ and let 
$$\X^\N:=\left\{ u: \N \rightarrow \X \right\}$$
Let $\nu$ be a measure on $\X$ and $P_t$ be a family of Markov transition kernels on $(\X, \rho)$. Using $P_t$ we can define the family of measures $(P_t^*\nu), \ t\ge 0$. Moreover, given an initial measure $\nu$ we can lift it to a measure $\nu P^\N$ on the pathspace $\X^\N$ using Kolmogorov's extension theorem. It is enough to define the family of finite time marginals on cylinder sets by $\nu P^n(A_1\times\ldots\times  A_n) =\int_{A_1}\ldots\int_{A_n}  P_1(x_{n-1}, dx_n)P_1(x_{n-2},x_{n-1})\ldots\nu(dx_1)$ for arbitrary $n\ge1$ and $A_i \in \B(\X)$.
\begin{definition}
	We say that a probability measure $\Gamma$ on $\X^\N \times \X^\N$ is an asymptotically equivalent coupling of two measures $m_1$ and $m_2$ on $\X^\N$ if $\Gamma\Pi_i^{-1}<< m_i$, for $i = 1,2$, where $\Pi_1(u,v) = u$ and
	$\Pi_2(u,v) = v$. We will write $\tilde C(m_1, m_2)$ for the set of all such asymptotically equivalent
	couplings.
\end{definition}
Given any bounded (measurable) function $\phi: \X \rightarrow \R$, we define $D _\phi \subset \X^\N \times \X^\N$ by
\begin{equation}
	D_\phi:=\left\{ (u, v )\in \X^\N \times \X^\N: \ \lim_{n\rightarrow\infty}\frac{1}{n}\sum_{k=1}^n \left(\phi(u_k) - \phi(v_k)\right) =0\right\}
\end{equation}
On the other hand a set $G$ of bounded, real-valued, measurable functions on $\X$ is said to
determine measures if, whenever $\mu_1,\mu_2 \in \mathcal{P}(\X)$ are such that $\int \phi d\mu_1 = \int \phi d\mu_2$ for all $\phi \in G$, then $\mu_1 =\mu_2$.
\begin{theorem}
	Let $G : \X \rightarrow R$ be collection of functions which determines measures. Assume that there exists a measurable $\X_0 \subset \X$ such that for any $(u_0,v_0) \in \X_0 \times  \X_0 $ and any $\phi \in G$ there exists a $\Gamma= \Gamma(u_0, v_0, \phi) \in \tilde C(\delta_{u_0}P^\N, \delta_{v_0}P^\N)$ such that $\Gamma( D_\phi)>0$. Then there exist at most one ergodic invariant measure $\mu$ for $P$ with $\mu(\X_0) > 0$. In particular if $\X_0 = \X$, then there exists at most one, and hence ergodic, invariant measure.
\end{theorem}
For application, it is useful to state a simple corollary. Consider a possibly different distance $\tilde\rho$ on $\X$ and set
$$D_{\tilde\rho}:=\left\{ (u, v )\in \X^\N \times \X^\N: \ \lim_{n\rightarrow\infty}\tilde \rho(u_n , v_n)=0\right\}.$$
Consider also 
$$\G_{\tilde\rho}:= \left\{\phi \in C_b(\X): \ \sup_{u\neq v} \frac{|\phi (u)- \phi(v)|}{\tilde \rho(u, v)} < \infty\right\}.$$
Then we have
\begin{corollary}\label{ergodicity_cor}
	Suppose that $\G_{\tilde\rho}$ determines measures on $(\X, \rho)$ and assume that $D_{\tilde\rho}$ is a measurable subset of $\X^N \times \X^N$. If $\X_0 \subset \X$ is a measurable set such that for each pair $(u_0,v_0) \in \X_0$ there exists an element $\Gamma \in \tilde C(\delta_{u_0}P^\N, \delta_{v_0}P^\N)$ such that $\Gamma( D_{\tilde \rho})>0$ ,then there exists at most one ergodic invariant measure $\mu$ with $\mu(\X_0) > 0$.
\end{corollary}

\section{Main results}\label{sec:main}
Recall that we work under \Cref{ass}. 
\begin{definition}\label{def-sol-H1}
	A stochastic process $\phi_t$ is a weak solution of \eqref{eq:main} if  $\phi\in C(0, T; V)\cap L^2(0, T; W)$ almost surely, it is adapted and for every $\psi  \in W$ it verifies $\PP-$a.s. for every $t\ge 0$
	\begin{align}\label{eq:hm_weak_def}
		\brak{\phi(t)-\phi_0, \psi} &+  \brak{\nabla \phi(t)-\nabla \phi_0, \nabla\psi} + \int_0^t\brak{A^{1/2}\phi(s), A^{1/2} \psi}ds \\ &+ \int_0^t \brak{\nabla^\perp \phi(s)\cdot \nabla v, \psi}ds + \int_0^t \brak{\Delta \phi(s), \nabla^\perp \phi(s)\cdot \nabla \psi}ds = \brak{ \psi, W_t},\notag
	\end{align}
	where we interpret $\brak{ \psi, W_t}=\prescript{}{V}{\langle}\psi, W_t \rangle\prescript{}{V{'}}{}$.
\end{definition}
\begin{theorem}\label{WP_H^1}
	For every $p>2$, $T>0$, $v \in W^{1,2}(\D)$ and $\phi_0\in L_{\F_0}^{2p}(\Omega, V)$ there exists a pathwise unique solution of \Cref{eq:main} in the sense of \Cref{def-sol-H1}. Moreover $\phi$ is a Markov process in $V$ and it satisfies,  
	\begin{align}
		&\EE\left[ \left\| \phi(t)\right\|_V^2 + 2\int_0^t|A^{1/2 }\phi(s)|^2ds\right] =  \EE\left[ \left\| \phi_0\right\|_V^2 \right] + \kappa t \label{eq:en_balance}\\
		&\EE\left[ \sup_{t\in [0, T]} \|\phi(t)\|_V^p + \left(\int_0^T |A^{1/2} \phi(t)|^2 dt\right)^{p/2}\right] \lesssim 1\label{eq:p_integrab}
	\end{align}
\end{theorem}
\begin{remark}
	The $p-$ integrability assumption in \Cref{WP_H^1} is needed only to show existence. Indeed uniqueness holds under the more natural assumption $\phi_0 \in L^2_{\F_0}(\Omega, V)$. 
\end{remark}
\begin{remark}
	For our needs, it is not required to work with probabilistically strong solutions, since we are only interested in properties of the laws and the whole analysis could be carried out with minor modifications without requiring this property. For this reason we only construct martingale solutions via a compactness argument and we omit the details of the steps required to show that solutions can be defined on the original probability space. Thanks to pathwise uniqueness, existence of strong solutions follows either from Yamada-Watanabe theorems (see \cite{kurtz_weak_2014},\cite{ondrejat_uniqueness_2004}, \cite{_concise_2007}, \cite{rockner_yamadawatanabe_2008}) or from Gyongy-Krylov criterium \cite{gyongy_existence_1996} which also provides convergence in probability of the approximations, on the original probability space (we refer the reader to \cite[Section 2.4]{flandoli_stochastic_2023} for the full argument carried out in a similar setting).
\end{remark}
The next theorem treats the case of \Cref{ass2} and provides an additional energy balance relation. 
\begin{theorem}\label{WP_H^2}
	Assume that $A$ satisfies \Cref{ass2}, $v\in Y$ and the noise coefficients $(f_k)_{k=1}^M \subset L^2(\D)$. Then for every $\phi_0 \in W^{2,2}(\mathcal{D}) \cap V$ let $\phi_t$ be the unique solution of \Cref{eq:main} in the sense of \Cref{def-sol-H1}. Then the process $\xi = \phi -\Delta \phi - v$ belongs to $C(0, T; L^2) \cap L^2(0, T; Y) $ almost surely and it verifies, for every $\psi \in Y$,
	\begin{align}\label{eq:weak_L^2}
		\brak{\xi_t, \psi} - \brak{\xi_0, \psi} = -\int_0^t \brak{B^{1/2 }\xi_s, B^{1/2} \psi}ds + \int_0^t\brak{ \xi_s, \nabla^\perp \phi \cdot \nabla \psi}ds + \int_0^t \brak{\xi_s , dW_t}.
	\end{align}
	Moreover it satisfies the following energy balance
	\begin{align}\label{eq:L2_en_bal}
		\EE[|\xi(t)|^2] - \EE[|\xi_0|^2]
		+2\int_0^t\EE\left[\|\xi\|_{Y}^2\right]dt = 2 \int_0^t\EE\left[\brak{B^{1/2} \xi, B^{1/2} v} \right]dt+ \sigma t,
	\end{align}
	where $\sigma = \sum_{k=1}^M \sigma_k^2|f_j|^2$. Finally, $\xi$ is a Markov process in $L^2$. 
\end{theorem}
\begin{theorem}\label{thm:ergod}\textbf{(Uniqueness of the invariant measure)}
	Assume \Cref{ass}. Fix $\kappa>0$. There exists $N=N(\kappa)>0$ such that if $(I-\Delta)P_N(V)\subset \operatorname{span}\left(\{f_1, \ldots f_M\}\right)$, and $\sum_{k=1}^M \sigma_k^2|(I-\Delta)^{-1/2} f_k|^2=\kappa$ then there exists a unique ergodic invariant measure $\mu$ for the equation \eqref{eq:main} on $V$, which is supported on $W$. Moreover the following energy balance holds
		\begin{equation}\label{V_balance_stat}
			\int_V\left|A^{1/2} \phi\right|^2 \mu(d\phi)= \frac{\kappa}{2}.
	\end{equation}
	If \Cref{ass2} holds, the unique invariant measure is supported on $(I-\Delta)^{-1}D(B^{1/2})$.
\end{theorem}
\begin{remark}
		The conditional assumptions on $\kappa$ and $N$ might seem a bit cumbersome, so let us add a brief comment: first of all, they are needed only to prove ergodicity, while existence of the invariant measure holds in general; on the other hand, when we perform the asymptotic coupling argument, we need to show that controlling a finite number of modes ($N>0$) of $\phi$ is sufficient to `enslave' all the remaining modes (cf. \Cref{prop:convergence-sol}). The number $N$ depends on how much instantaneous energy $\kappa=\sum_{k=1}^M \sigma_k^2|(I-\Delta)^{-1/2} f_k|^2$ we are injecting, but not on the specific choice of the sequence $(\sigma_k)_{k=1}^M$. Thus we can fix $\kappa>0$ and determine $N(\kappa)$ independently of $M$. Only then the argument requires that the noise is sufficiently coloured to act on all the `determining modes' and we have included this condition as an hypothesis of the theorem. It is only when we want to check that such hypothesis is not empty, that we must make a choice of $M>0$ and $(f_k)_{k\ge 1}$. Once $M=M(\kappa)$ is known, it is sufficient to choose the sequence $(\sigma_k)_{k=1}^M$ appropriately. One example of such a sequence can be obtain by setting $f_k$ as in \Cref{rem:noise_choice} and $\sigma_k:= \sqrt{\kappa/N}, \ k=1\ldots N$.
\end{remark}

\begin{remark}
	If we choose $A=A_0$ from \Cref{example_op}, we can see that the following additional energy balance holds for the invariant measure
	\begin{equation}\label{H2_balance_stat}
			\int_V \beta|\nabla\phi|^2 + (\beta+1) |\Delta \phi|^2 +  |\nabla  \Delta \phi |^2\mu(d\phi) = \frac{\sigma}{2} + \int_V \brak{\nabla^\perp \phi \cdot \nabla v, \Delta \phi}\mu(d\phi). 
	\end{equation}
\end{remark}
\section{Proofs}
\subsection{Well-posedness}\label{sec:WP}
We begin by proving the pathwise uniqueness of solutions. 
\begin{proof}[Proof of \Cref{WP_H^1}   \ (Uniqueness)]\label{proof_uniq}
	Let $\phi, \psi$ be two solutions in the sense of \Cref{def-sol-H1} starting from two initial data $\phi_0, \psi_0$ and consider $m(t)= \psi(t)- \phi(t)$. It satisfies, in the usual weak sense 
	\begin{equation*}
		\partial_t (m - \Delta m) + Am + \NN(\psi, v-\Delta \psi) -\NN(\phi, v-\Delta \phi) = 0
	\end{equation*}
	Thanks to the regularities of $\phi, \psi$ we can see that $$(I-\Delta)^{1/2}m \in L^\infty(0, T; L^2) \cap L^2(0, T; D(A^{1/2}(I-\Delta)^{-1/2}))$$
	and thanks to the properties of the non-linear operator $\mathcal{N}$ from \Cref{N_properties} we also obtain 
	$$\partial_t (I-\Delta)^{1/2}m \in L^2(0, T; D(A^{1/2}(I-\Delta)^{-1/2})')$$
	Thus the classical Lions-Magenes lemma \cite[Chapter III, Lemma 1.2]{temam_navierstokes_}, yields that $m$ is almost everywhere equal to a continuous function into $V$ and
	$$\frac{d}{dt}\|m\|_V^2  + 2|A^{1/2} m |^2 + \brak{\NN(m, \Delta m), \phi} =0$$
	where in order to treat the non-linear terms we have used the properties of $\NN$ from \Cref{N_properties}:
	\begin{align}\label{eq:nonlinear_trick}
		&\brak{-\NN(\psi, \Delta \psi) +\NN(\phi, \Delta \phi), \psi -\phi }\notag\\
		&= \brak{\NN(\psi, \Delta \psi) ,\phi } + \brak{\NN(\phi, \Delta \phi), \psi  }\notag\\
		&= -\brak{\NN(\psi, \phi) ,\Delta \psi } - \brak{\NN(\phi, \psi),\Delta \phi  }\notag\\
		&= -\brak{\NN(\psi, \phi) ,\Delta m} \notag\\
		&= -\brak{\NN(m, \phi) ,\Delta m}\notag \\
		&= \brak{\NN(m, \Delta m), \phi}
	\end{align}
	Now we estimate, again thanks to \Cref{N_properties}, Young's inequality and the embedding $W\hookrightarrow V^2$ which holds with constant $C$
	\begin{align*}
		\brak{\NN(m, \Delta m), \phi}
		&\le \|\phi\|_{V^2} \|m\|_{V^2} |\nabla m| \\
		&\le \frac{1}{4}  \|m\|_{W} ^2 + C'\|\phi\|_{W}^2 |\nabla m|^2.
	\end{align*}
	We have obtained 
	$$\frac{d}{dt}\|m\|_V^2 + |A^{1/2} m|^2 \le C' \|\phi\|_W^2 |\nabla m|^2.$$
	Since $ \|\phi\|^2_W \in L^1(0, T)$ almost surely, it follows by Gronwall that 
	\begin{align}\label{eq:gronwall_uniq}
		\|m(t)\|^2_V\le \|m(0)\|^2_V\exp\left(C'\int_0^t \|\phi(s)\|_W^2 ds\right),
	\end{align}
	for every $t>0$ and almost surely. If $\phi_0= \psi_0$ a.s., then it follows $m(t)=0$ for all $t\ge 0$ $\PP-$a.s., i.e. $\phi$ and $\psi$ are indistinguishable. 
\end{proof}
We are now ready to construct solutions. 
Recall the definition of the ONB $\{e_k\}_{k\ge 1}$ of $V$ constructed in \Cref{rem:onb} and of $P_N$, the orthogonal projection, with respect to $\brak{\cdot , \cdot}_V$ onto $V_N = \operatorname{span}\{e_1, \ldots e_N\}$. Recall in particular $\brak{(I-\Delta)e_k, e_j}= \brak{e_k, e_j}_V= \delta_{kj}$ and the relations 
$$\brak{x, e_k}= \brak{(I-\Delta)^{-1}x, e_k}_V, \qquad \sum_{k=1}^N\brak{x, e_k}e_k= P_N(I-\Delta)^{-1}x, \qquad \forall x \in L^2(\D). $$
Moreover, recalling that $P_N$ can be naturally extended to $L^2$, we denote with $(P_N)^*$ the adjoint of $P_N$ with respect to the $L^2$ inner product. It holds $(P_N)^*= (I-\Delta)P_N(I-\Delta)^{-1}$, however we shall never use this explicit expression.
Consider the finite dimensional approximation
$$\phi^n(t,x)= \sum_{k=1}^ng^n_k(t)e_k(x) .$$
We set $g^n_k(0)= \brak{\phi_0, e_k}_V$ and $g^n_k(t)$ solving the system of SDEs
\begin{equation}\label{eq:gal_def}
	dg^n_k+ \brak{\NN(\phi^n, v- \Delta \phi^n), e_k} + \lambda_k g^n_k= \sum_{j=1}^M \sigma_j \brak{f_j, e_k}dW_t^j, \qquad \{k=1,\ldots, n\}.
\end{equation}
Expanding
$$\brak{\NN(\phi^n, v- \Delta \phi^n), e_k} = \sum_{p=1}^n g_p^n(t)d_{p,k} + \sum_{p=1}^n\sum_{m=1}^n g_p^n(t)g_m^n(t) j_{p,m,k}$$
where
$$d_{p,k}= \brak{\nabla^\perp e_p\cdot \nabla v, e_k}, \qquad j_{p,m,k}=\brak{\nabla^\perp e_p\cdot \nabla e_{m}, \Delta e_k}, $$
we see that the vector $g^n = (g_k)_{k=1}^n$ solves a closed finite dimensional SDE with a quadratic nonlinearity, thus by the classical theory of SDEs, existence of local probabilistically strong solutions is guaranteed. 
Thanks to this definition of $g_k$ we can check that $\phi^n$ solves
\begin{equation}\label{eq:phi_n}
	d(\phi^n - \Delta \phi^n) + (P_n)^*\NN(\phi^n, v- \Delta \phi^n)dt + A\phi^n dt = (P_n)^*dW_t.\end{equation}
Indeed $d(\phi^n - \Delta \phi^n) = \sum_{k=1}^n dg^n_k (I-\Delta)e_k$
and for every $x\in L^2$, $ y \in V$,
$$\brak{\sum_{k=1}^n\brak{x, e_k} (I-\Delta)e_k, y} = \sum_{k=1}^n\brak{x, e_k}\brak{y, e_k}_V = \brak{x, P_n y},$$
while 
$$\sum_{k=1}^n\brak{\lambda_k g_k^n(I-\Delta) e_k, y} =\sum_{k=1}^ng_k^n\lambda_k\brak{e_k, y}_V = \brak{A\sum_{k=1}^ng_k^ne_k, y} = \brak{A\phi^n, y}. $$
Now we need to prove some estimates:\\
\underline{Step 1. It\^{o} formula.}\\
Applying the finite dimensional It\^{o} formula to $\|\phi^n(t)\|_V^2 = \sum_{k=1}^n g_k^2(t)$, since
$$\sum_{k=1}^n\brak{\NN(\phi^n, v- \Delta \phi^n), e_k}g_k(t) = \brak{\NN(\phi^n, v- \Delta \phi^n), \phi^n}=0, $$
and since the quadratic variation of $g_k$ is $\kappa_k=\sum_{j= 1}^M\sigma_j^{2}\brak{f_j,e_k}^2$, we obtain
\begin{equation}\label{gal-ito}
	d\|\phi^n\|_V^2   + 2\brak{A\phi^n, \phi^n}dt = 2\brak{\phi^n, dW_t} + \kappa^ndt.
\end{equation}
Where $\kappa^n := \sum_{k\le  n} \kappa_k$. Clearly we have that $\kappa^n$ is increasing in $n$ and $\kappa^n \uparrow \kappa=\sum_{k=1}^M \sigma_k^2|(I-\Delta)^{-1/2} f_k|^2 $, 
For every $R > 0$ we introduce the stopping time $\tau_R = \inf \{ t > 0 : \|\phi^n(t)\|_V^2 > R \}$ or equal to $T$ in case the set is empty.
We consider \eqref{gal-ito} integrated from $t=0$ to $t = s \wedge \tau_R$
\begin{align} \label{EQ:ito_H01}
	\| \phi^n(s \wedge \tau_R) \|_V^2 &+ 2 \int_0^s \mathbf{1}_{r \leq \tau_R} \| \phi^n(r) \|_W^2 dr
	= \| \phi^n(0) \|_V^2 \notag\\
	&+ \kappa^n (s \wedge \tau_R) + 2 \sum_{k = 1}^{M} \int_0^s \mathbf{1}_{r \leq \tau_R} \brak{\phi^n(r), f_k} \sigma_k dW_r^k.
\end{align}
Then, since $\EE \bra{\int_0^T \mathbf{1}_{r \leq \tau_R} \brak{\phi^n(r), f_k}^2 \sigma_k^2 dr} \leq RT\kappa $, the It\^{o} integrals of the last identity are true martingales.
Applying BDG inequality and Cauchy-Schwartz we estimate
\begin{align*}
	\EE \bra{\sup_{0 \leq s \leq t} \left|\int_0^s \mathbf{1}_{r \leq \tau_R} \sum_{k = 1}^{M } \sigma_k \brak{ \phi^n(r), f_k } d W^k_r\right|}
	&\leq 2 \EE \bra{ \int_0^t \mathbf{1}_{r \leq \tau_R} \left( \sum_{k = 1}^{M } \sigma_k \brak{ \phi^n(r), f_k } \right)^2 dr}^{1/2}\\
	&\leq 2 \kappa^{1/2} \EE \bra{ \int_0^t \mathbf{1}_{r \leq \tau_R} \|\phi^n(r)\|_V^2 dr}^{1/2},
\end{align*}
which, after identifying
\begin{equation*}
	\EE \bra{ \sup_{0 \leq s \leq t } \|  \phi^n(s \wedge \tau_R) \|_V^2} = \EE \bra{ \sup_{0 \leq s \leq t } \| \phi^n(s) \mathbf{1}_{s \leq \tau_R} \|_V^2},
\end{equation*}
and applying Young's inequality, provides
\begin{align*}
	\EE \bra{ \sup_{0 \leq s \leq t } \| \phi^n(s) \mathbf{1}_{s \leq \tau_R} \|_V^2}
	&\leq \EE \bra{ \|  \phi^n(0) \|_V^2 } + \kappa (T +2) \\
	&+ \EE \bra{ \int_0^t \mathbf{1}_{r \leq \tau_R} \|\phi^n(r)\|_V^2 dr}. 
\end{align*}
By Gronwall lemma it follows that
\begin{eqnarray*}
	\EE \bra{ \sup_{0 \leq s \leq t } \| \phi^n(s) \mathbf{1}_{s \leq \tau_R} \|_V^2} \leq C,
\end{eqnarray*}
where $C$ is a constant independent from $R$ and $n$.
Taking the limit $R \to \infty$, by the monotone convergence theorem we get
\begin{eqnarray*}
	\EE \bra{ \sup_{0 \leq s \leq T } \| \phi^n(s) \|_V^2} \leq C.
\end{eqnarray*}
This implies that $\phi^n \in C_\mathcal{F}([0,T];V)$ and $\EE \bra{ \int_0^T \brak{ \phi^n(s), f_k }^2 ds } \leq \infty$.
Therefore, $t \mapsto \int_0^t \brak{ \phi^n(s), dW_s }$ is a martingale, and so its expected value is zero.
Consequently, the RHS of~\eqref{EQ:ito_H01} and $\| \phi^n(t) \|_V^2$ have finite expected value, and hence the same is true for the other term on the LHS, namely
\begin{equation*}
	\EE \bra{ \int_0^T \| \phi^n(s) \|_W^2 ds } \leq C.
\end{equation*}
Particularly, we get the energy relation
\begin{align}\label{eq:ex-ito-L2}
	\EE \bra{ \| \phi^n(t) \|_V^2} + 2 \EE \bra{ \int_0^t \| \phi^n(s) \|_W^2 ds } 
	= \EE \bra{ \| \phi^n(0) \|_V^2 } + \kappa t.
\end{align}
Consequently, the solution $\phi^n$ is global in time and lives in a bounded set, uniform in $n$ of $L^{\infty}(0, T; V)\cap L^2(0, T, W)$. 
The next step is to prove some higher moments estimate required to get some uniform bound on time. \\
\underline{Step 2. $L^p(\Omega)$-estimate}\\
From now on we assume $\phi_0 \in L^{2p}_{\F_0}(\Omega;  V)$. Taking the supremum in time of \eqref{gal-ito} we have.
\begin{align}\label{ito_sup}
	\sup_{t\in [0, T]}\|\phi^n(t)\|_V^2  &+ 2\int_0^T|A^{1/2} \phi^n(s)|^2ds  \notag \\&\le  \|\phi^n_0\|_V^2  + \kappa^n T + 2 \sup_{t\in [0, T]} \left|\int_0^t\brak{\phi^n(s), dW_s} \right|. 
\end{align}
Raising the LHS to the power $p$ we get\footnote{we use that for $A, B >0$, $A^p+B^p \le (A + B)^p \le 2^{p-1}( A^p + B^p)$}
\begin{align}\label{eq:ito-p-mom}
	\sup_{t\in [0, T]}\|\phi^n(t)\|_V^{2p}  &+ 2^p\left(\int_0^T|A^{1/2} \phi^n(s)|^2ds\right)^p  \notag\\ \lesssim & \|\phi^n_0\|_V^{2p} + (\kappa^n)^p T^p + 2^p \left(\sup_{t\in [0, T]} \left|\int_0^t\brak{\phi^n(s), dW_s} \right|\right)^p.
\end{align}
Up to another localization argument as in the previous step, Burkholder-Davis-Gundy inequality gives
\begin{align*}
	\EE\left[\sup_{t\in [0, T]} \left|\int_0^t\brak{\phi^n(s), dW_s}\right|^p\right] &\le C_p\EE\left[ \left(\int_0^T \sum_{k=1}^N\sigma_k^2 \brak{\phi^n(s), f_k}^2 ds \right)^{p/2}\right] \\
	&\le C_p\EE\left[ \left(\int_0^T \sum_{k=1}^N\sigma_k^2 \brak{(I-\Delta)^{1/2}\phi^n(s), (I-\Delta)^{-1/2}f_k}^2 ds\right)^{p/2} \right] \\
	&\le C_p\EE\left[ \left(\int_0^T \sum_{k=1}^N\sigma_k^2 \|\phi^n(s)\|_V^2 |(I-\Delta)^{-1/2}f_k|^2 ds \right)^{p/2}\right] \\
	&\le C_pT^{p/2}\kappa ^{p/2}\EE\left[ \sup_{t\in [0,T]}\|\phi^n(t)\|_V^{p} \right] \\
	&\le {\tilde C_p}T^{p}\kappa ^{p} + \frac{1}{2}\EE\left[ \sup_{t\in [0,T]}\|\phi^n(t)\|_V^{2p} \right]
\end{align*}
where we have used in the last step Young's and Jensen's inequality. Thus, taking the expected value of \eqref{eq:ito-p-mom} we obtain 
\begin{align}\label{eq:gal-p-int}
	&\EE\left[\sup_{t\in [0, T]} \|\phi^n(t)\|_V^{2p}  + \left(\int_0^T \|\phi^n(s)\|_V^2 ds\right)^p \right] \lesssim_{\kappa, p, T} 1.
\end{align}
\underline{Step 3: Time regularity}\\
The last ingredient is the proof of the following estimate
\begin{align}\label{eq:time-regularity}
	\EE\left[ \left\| (I-\Delta)\left(\phi^n(t)- \phi^n(s)\right)\right\|^p_{W'}\right] \lesssim (t-s)^{p/2}.
\end{align}
We have
\begin{align*}
	(I-\Delta)\left(\phi^n(t)- \phi^n(s)\right) &=\int_s^t A\phi^n(r)dr -\int_s^t \nabla^\perp \phi^n(r)\cdot \nabla vdr\\& +\int_s^t  \nabla^\perp \phi^n(r)\cdot \Delta \phi^n(r) dr + \sum_{k=1}^n \sigma_kf_k\left(W^k_t - W^k_s\right).
\end{align*}
First we estimate
$$\left\|\int_s^t A\phi^n(r)dr\right\|_{W'}  \le \sqrt{t-s}\left(\int_0^t |A^{1/2}\phi^n(r)|^2dr\right)^{1/2} $$
Next, thanks to the embedding $W \hookrightarrow V^2$ we have, for every $\psi \in W$
$$\brak{\nabla^\perp \phi^n(s)\cdot \nabla v, \psi} \le \|\psi\|_{\infty} |\nabla v| |\nabla \phi^n| \lesssim \|\psi\|_{W} \|\nabla v\|_\infty | A^{1/2}\phi^n|,$$
thus 
\begin{align*}
	\left\|\int_s^t \nabla^\perp \phi^n(r)\cdot \nabla v dsr\right\|_{W'}  \lesssim  (t-s)^{1/2}|\nabla v|\left(\int_0^t |A^{1/2}\phi(r)^n|^2 dr\right)^{1/2}.
\end{align*}
On the other hand, since for every $\psi \in V$ it holds 
\begin{align*}
	\brak{\nabla^\perp \phi^n(s)\cdot \nabla \Delta \phi^n(s), \psi}&= \brak{\nabla^\perp  \phi^n(s)\otimes  \nabla^\perp \phi^n(s), \nabla \otimes\nabla^\perp \psi}  \\
	&\lesssim  |\nabla^\perp  \phi^n(s)|^2_4 |\nabla \otimes\nabla^\perp \psi| \\
	&\lesssim   |\nabla^\perp  \phi^n| \|\phi^n\|_{V^2}\|\psi\|_{V^2} \\
	&\lesssim   |\nabla^\perp  \phi^n| \|\phi^n\|_{W} \|\psi\|_{W},
\end{align*}
we get 
$$ \left\| \int_s^t  \nabla^\perp \phi^n(s)\cdot \nabla \Delta \phi^n(s)  ds\right\|_{D(A^{-1/2})}\lesssim \sqrt{(t-s)}\left(\sup_{[0, t]}|\nabla^\perp  \phi^n(s)|\right)\left(\int_0^t  |A^{1/2}\phi^n(s)|^2ds\right)^{1/2},$$
from which 
\begin{align*}
	&\EE\left[\left\| \int_s^t  \nabla^\perp \phi^n(s)\cdot \nabla \Delta \phi^n(s)  ds\right\|^p_{D(A^{-1/2})}\right]\\ &\lesssim (t-s)^{p/2}\EE\left[\sup_{[0, t]}|\nabla^\perp  \phi^n(s)|^{2p}\right]^{1/2}\EE\left[\left(\int_0^t  |A^{1/2}\phi^n(s)|^2ds\right)^{p}\right]^{1/2}.
\end{align*}
Finally, for every $p\ge 2$,
$$\EE\left[\left\| \sum_k \sigma_k f_k(W_t^k- W_s^k)\right\|^p_{D(A^{-1/2})}\right] \lesssim (t-s)^{p/2}. $$
Putting all together and using \Cref{eq:gal-p-int}, we get 
\begin{align*}
	\EE\left[ \left\| (I-\Delta)\left(\phi^n(t)- \phi^n(s)\right)\right\|^p_{D(A^{-1/2})}\right] \lesssim  (t-s)^{p/2}.        \end{align*}
Thanks to this estimate, we obtain
$$\EE\left[\|\phi^n - \Delta \phi^n\|^p_{W^{\alpha, p}(0, T; W')}\right] \lesssim \int_0^T\int_0^T |t-s|^{p/2-1  - \alpha p} <\infty \iff \alpha < 1/2.$$
That is, $(t\rightarrow(I-\Delta)\phi(t)) \in W^{\alpha,p}(0, T; D(A^{-1/2}))$ almost surely for every $p>1$ and $\alpha < 1/2$.
Putting everything together, we have proven the following
\begin{proposition}\label{prop:gal_est}
	Let $\phi_0 \in L^{2p}_{\F_0}(V)$ for some $p > 1$. Then the family of solution $\{\phi^n\}_{n\ge 1}$ of the finite dimensional systems \eqref{eq:gal_def} is uniformly bounded in 
	$$L^{2p}(\Omega, L^\infty_{w-*}(0, T; V))\cap L^p(\Omega; L^2( 0, T; W))\cap L^p_\F(W^{\alpha, p}(0, T, V')) $$
	for every $\alpha < 1/2$\footnote{The space $L^\infty_{w-*}(0, T; V)$ is the space of bounded functions with values in $V$ endowed with the weak$^*$ topology. Since the same space with the strong topology is not separable, we use the weak$^*$ topology to define $p-$integrable random variables. Recall also that the Borel sets relative to the weak$^*$ topology form a strict subset of those generated by the strong topology \cite{nottalagrand}.} .
\end{proposition}
We have now all the ingredients to construct our solution of \eqref{eq:main}.
\begin{proof}[Proof of \Cref{WP_H^1}, (Existence)]\label{proof_ex}
	The proof employs a compactness argument. For any $p > 1$ and $\alpha < 1/2$, thanks to the embedding $W \hookrightarrow V^2$, we have by a simple interpolation $\|x\|_{V}\le \|x\|^{2/3}_{W}\|x\|^{1/3}_{V'}$, thus for $\theta=1/3$, and $s_*=(\alpha p -1 -p)/(3p)$, so that $-1/s_* >2$, \cite[Corollary 9]{simon_compact_1987} implies that the embedding $L^2(0, T; W)\cap W^{\alpha,p}(0, T, V')\rightarrow  L^2(0, T; V)$ is compact. By \cite[Theorem 2.2]{flandoli_martingale_1995}, for $s>1$ and choosing now $\alpha<1/2$ such that $\alpha p >1$, we have also the compact embedding $L^2(0, T; W)\cap W^{\alpha,p}(0, T, V')\rightarrow C(0, T; V^{-s})$. Consider the space $\X=C(0, T; \R^\N)$ where $\R^\N$ is equipped with the distance $d((g_k)_{k\ge 1}, (h_k)_{k\ge 1}):= \sum_{k=1}^{+\infty} 2^{-k}(|g_k -h_k|\land 1)$ (which makes it a Polish space). 
	Thanks to the uniform estimates provided by \Cref{prop:gal_est} and combining the compact embeddings from above, the laws of $(\phi^n, (W_t^k)_{k\ge 1})$ are tight on $L^2(0, T;V) \cap C(0, T; V^{-s})\times \X$, thus, by Prokhorov's theorem, we can extract a subsequence (not relabelled) $\mu^{n}= \L((\phi^n, (W_t^k)_{k\ge 1}))$ and a probability measure $\mu$ on $L^2(0, T; V)\cap C(0, T; V^{-s})\times \X$ such that $\mu^n \weakstarto \mu$ in $\mathcal{P}(L^2(0, T; V)\cap C(0, T; V^{-s})\times \X)$. Thanks to Skorokhod's representation theorem, there exists a new probability space $(\tilde \Omega, \tilde \F, \tilde \F_t, \tilde \PP)$ and new processes $\tilde \phi^n, (\tilde W^{k,n}_t)_{k\ge 1}$ and $\tilde \phi, (\tilde W_t)_{k\ge 1}$ such that $(\tilde W_t)_{k\ge 1}$ is a family of standard brownian motions, $\tilde \phi^n$ still solves \eqref{eq:phi_n} with respect to ${(\tilde W^t_k)}_{k\ge 1}$ and 
	\begin{align*}
		\tilde \phi^n &\rightarrow \tilde \phi \qquad \qquad \quad \text{in } L^2(0, T;V) \cap C(0, T; V^{-s}) \quad   \tilde \PP-a.s. \\
		(W^{k,n}_t)_{k\ge 1} &\rightarrow ({\tilde W}^{k}_t)_{k\ge 1} \qquad \text{in } \X \quad  \tilde \PP-a.s.
	\end{align*}
	Moreover, up to extracting other subsequences, we can assume that 
	\begin{align*}
		\tilde \phi^n \rightharpoonup \tilde \phi \qquad &\text{in } L^2(0, T;W) \quad  \tilde \PP-a.s. \\
		\tilde \phi^n \weakstarto  \tilde \phi \qquad &\text{in } L^\infty(0, T;V) \quad  \tilde \PP-a.s. 
	\end{align*}
	Recalling the relation \eqref{eq:phi_n}, we have that for any test function $\psi \in W$, 
	\begin{align*}
		\brak{\tilde\phi^n_t, \psi} + \brak{\nabla \tilde\phi^n_0, \nabla \psi} &-\brak{\tilde\phi_0^n, \psi} + \brak{\nabla \tilde\phi^n_t, \nabla \psi} + \int _0^t \brak{\NN(\tilde\phi^n, v- \Delta \tilde\phi^n), P_n\psi}ds\\ &= -\int _0^t \brak{A^{1/2}\tilde\phi^n, A^{1/2}\psi}ds   + \brak{\tilde W_t, P_n \psi}
	\end{align*}
	Recalling $P_n x \rightarrow x$ strongly in $H^1_0$ and owning to the above stated convergences of $\tilde \phi^n$, we can pass to the limit in each term obtaining that $\tilde \phi$ verifies \eqref{eq:hm_weak_def}. The only delicate term is the nonlinearity which we treat explicitly: thanks to the almost sure strong convergence of $\tilde \phi^n$ in $L^2(0, T;V)$ we have that $\nabla^\perp \tilde\phi^n \otimes \nabla^\perp \tilde\phi^n$ almost surely converges strongly to $\nabla^\perp \tilde\phi \otimes \nabla^\perp \tilde\phi$ in $L^1(0, T; L^1)$. Recalling \eqref{N-properties 2} in \Cref{N_properties} and that $P_n \psi \rightarrow \psi$ also strongly in $W\hookrightarrow V^2$ we have that 
	\begin{align*}
		\int_0^t\brak{\NN(\tilde\phi^n, \Delta \tilde\phi^n), P_n\psi}ds&= \int_0^t \brak{\nabla^\perp \tilde \phi^n \otimes \nabla^\perp \tilde \phi^n , \nabla \otimes \nabla^\perp  P_n \psi}\\
		&\rightarrow \int_0^t \brak{\nabla^\perp \tilde \phi \otimes \nabla^\perp \tilde \phi : \nabla \nabla^\perp \psi} \\
		&=-\int_0^t\brak{\Delta \tilde\phi,\nabla^\perp\tilde\phi\cdot \nabla \psi}ds.
	\end{align*}
	In order to complete the proof we are left to prove \eqref{eq:en_balance} and \eqref{eq:p_integrab} hold and that $\tilde\phi$ have continuous paths in $V$. The proof of the Markovianity instead is deferred to \Cref{feller}. 
	We begin with an It\^{o} formula. 
	\begin{proposition}\label{ito-formula}\textbf{(It\^{o}-Formula for $\|\phi\|^2_{V}$})
		Let $\phi(t)$ be a solution of \eqref{eq:main} in the sense of \Cref{def-sol-H1}, then it holds 
		\begin{align*}
			|\phi(t)|^2 + |\nabla \phi(t)|^2 + 2\int_0^t|A^{1/2} \phi(s)|^2ds = |\phi_0|^2 + |\nabla \phi_0|^2 + \kappa t + 2\int_0^t\brak{\phi(s), dW_s},
		\end{align*}
		where $\kappa := \sum_k \sigma_k^2|(I-\Delta)^{-1/2}f_k|^2$.
	\end{proposition}
	\begin{proof}(sketch) 
		The proof is just an application of the It\^{o} formula. We either use an abstract infinite dimensional version like \cite[Theorem 2.13]{rozovsky_stochastic_2018} or we can prove it by hand as for instance in \cite[Theorem 2.4]{flandoli_stochastic_2023}). Namely, we consider a finite dimensional projection on $V_n$, $\phi^n := P_n \phi$ which solves
		\begin{equation*}
			\partial_t (\phi^n- \Delta \phi^n) + P_n\NN(\phi, \Delta \phi) + A\phi^n=   P_ndW
		\end{equation*}
		We then apply the finite dimensional It\^{o} formula on $\|\phi^n\|_V^2$
		and thanks to the convergence property of the projection, $P_n \phi \rightarrow \phi$ in $W$, and the continuity property of the nonlinearity, we obtain $\brak{ \NN(\phi, \Delta \phi), \phi^n} \rightarrow 0,$ 
		and we can pass to the limit the finite dimensional relation to obtain the desired identity. 
	\end{proof}
	Now the continuity of the trajectories in $V$ follows by combining the weak continuity provided by the embedding $L^\infty(0, T, V)\cap C(0, T; V^{-s}) \hookrightarrow C_w(0, T; V)$ with the continuity of the $V$ norm provided by the It\^{o} formula in \Cref{ito-formula}.  
	Finally we have
	\begin{proposition}\label{prop:p_moments}
		Let $\phi_0 \in L^p_{\F_0}(\Omega;  V)$. The unique solution of \Cref{eq:main} in the sense of \Cref{def-sol-H1} satisfies, for every $p\ge 1$ 
		\begin{align}
			&\EE\left[\sup_{t\in [0, T]} (|\phi(t)|^{2p} + |\nabla\phi(t)|^{2p}) + \left(\int_0^T \|\phi(s)\|_V^2 ds\right)^p \right] \lesssim 1.
		\end{align}
		Moreover, for every $\alpha < 1/2$ it holds 
		$$\EE\left[ \|(I-\Delta)\phi\|_{W^{\alpha,p}(0, T; W')}\right] \le C.$$
	\end{proposition}
	\begin{proof}
		Thanks to \Cref{ito-formula}, we can repeat verbatim by the same steps made to prove the analogous bounds for the Galerkin approximation. 
	\end{proof}
	Thanks to this last proposition, the proof of \Cref{WP_H^1} is complete. 
\end{proof}
\begin{proof}[Proof of \Cref{WP_H^2}]
	From \Cref{WP_H^1} and \Cref{ass2} it is easy to see that $\xi= \phi- \Delta\phi - v$ satisfies the weak formulation in \Cref{eq:weak_L^2}. Applying the It\^{o} formula to $|\xi|^2$ in the same fashion as in the proof of \Cref{ito-formula}, we immediately obtain the energy balance \eqref{eq:L2_en_bal}. The markovianity and the Feller property of $\xi$ can be proved by the same steps as we will do for $\phi$ in \Cref{feller}.
\end{proof}
\subsection{Feller property and existence of an invariant measure}
\begin{proposition}\label{prop:cont-init-data}
	Let $(\phi_0^n)_{n\ge 0}$ be a sequence in $L^2_{\F_0} (V)$ such that $\phi_0^n \rightarrow \phi_0$ a.s. in $V$. Let $\phi(t; \phi_0^n)$ and $\phi(t; \phi_0)$ denote the solutions to \Cref{eq:main} starting respectively from $\phi_0^n$ and $\phi_0$. Then $\phi(t; \phi_0^n) \rightarrow \phi(t, \phi_0)$ in $L^\infty(0, T; V)$ $\PP-a.s.$ . 
\end{proposition}
\begin{proof}
	Repeating the same steps as in the proof of uniqueness in \Cref{proof_uniq}, after setting $m^n(t)= \phi(t, {\phi^n_0}) - \phi(t, {\phi_0})$,  we arrive at \eqref{eq:gronwall_uniq}. Taking the supremum in time we get
	\begin{align}
		\sup_{t\in [0, T]} \|\phi^n(t, \phi^n_0)- \phi(t, \phi_0)\|_V\le \|\phi^n_0- \phi_0\|_V^2 \exp\left(C'\int_0^T |\phi(s, \phi_0)|_V^2 ds\right). 
	\end{align}
	Since $\PP-a.s.$ $\int_0^T |\phi(s, \phi_0)|_V^2 ds < \infty$, by letting $n \rightarrow\infty$ we get $\phi^n(t, \phi^n_0)- \phi(t, \phi_0) \rightarrow 0$ in $L^{\infty}(0, T; V)$ $\PP-a.s.$. 
\end{proof}
\begin{remark}
	The proof of this lemma uses the fact that $\phi, \phi^n$ can be defined on the same probability space and in particular with respect to the same noise $W_t$, i.e. we are in a setting were existence of probabilistically strong solutions holds. Since we have omitted the proof of this fact, let us briefly comment on how to prove \Cref{prop:cont-init-data} without making use of probabilistically strong solutions. It is enough to consider the sequence of laws $\L (\phi^n)$ and thanks to the a priori estimates of \Cref{sec:WP}, one can employ the same compactness argument as in \Cref{proof_ex} and passage to the limit in the weak formulation to show that $\phi^n$ converges almost surely, on an auxiliary probability space, to $\phi$ solving \eqref{eq:main} . 
	Thus the thesis holds, up to a change of probability space, which however is irrelevant in obtaining the Feller property in the next corollary. 
\end{remark}
\begin{corollary}\label{feller}
	The process $\phi(t)$ is Markov in V and its transition semigroup, defined as $P_t\Psi(\phi_0):= \EE\left[ \Psi(\phi(t; \phi_0))\right]$ for any $\phi_0 \in V$ is Feller, namely $P_t: C_b(V)\rightarrow C_b(V)$. 
\end{corollary}
\begin{proof}
	Thanks to \Cref{prop:cont-init-data}, we are left to prove only the Markov property. For every $t\ge t_0 \ge 0$, let $\phi(t,t_0, \phi_0)$ be the solution of \Cref{eq:main} with initial condition $\phi_0 \in L^{2p}_{\F_{t_0}}(\Omega; V)$ given at time $t_0$ and defined analogously as in  \Cref{def-sol-H1} and set 
		$$P_{t_0,t}(\Psi)({\phi_0}) = \EE\left[\Psi(\phi (t, t_0, \phi_0))\right]$$
	We have to show that for all $ \Psi \in B(V),$ and for all $   0 \le t_0 \le s \le t$
	\begin{align}\label{eq:markov_prop}
		\EE\left[ \Psi(\phi(t, t_0; \phi_0))| \F_s\right]= P_{s,t}(\Psi)(\phi(s, t_0, \phi_0)).
	\end{align}
	The uniqueness proved in \Cref{proof_uniq} provides 
	$$\phi(t, t_0;  \phi_0)= \phi(t, s; \phi(s, t_0; \phi_0)).$$
	Denote $\eta:=\phi(s, t_0; x_0)$, then \eqref{eq:markov_prop} can be rewritten as $  \EE\left[ \Psi(\phi(t; s, \eta))| \F_s\right]= P_{s,t}(\Psi)(\eta)$.
	It is enough to prove this identity for a generic square integrable $\F_s-$measurable random variable $\eta$. Moreover we can assume $\Psi \in C_b(V)$. 
	We assume that $\eta$ takes only a finite number of values, i.e.: 
	$$\eta = \sum_{j=1}^N \eta_j \one _{\Gamma_j},$$
	where $\Gamma_1, \ldots, \Gamma_N \subset \F_s$ is a partition of $\Omega$ and $\eta_j \in V $. Then by uniqueness, $\phi(t, s, \eta)= \sum _{j=1}^n \phi(t, s, \eta_j)\one _{\Gamma_j}$, thus 
	$\EE\left[ \Psi(\phi(t, s; \eta))| \F_s\right] = \sum_{j=1}^n \EE\left[ \Psi(\phi(t, s; \eta_j)) \one _{\Gamma_j}| \F_s\right]$
	but now since $\eta_j$ is deterministic, $\phi(t, s; \eta_j)$ is independent of $\F_s$ while $\one _{\Gamma_j}$ is measurable with respect to $\F_s$, the Freezing lemma provides
	$\sum_{j=1}^n \EE\left[ \Psi(\phi(t, s; \eta_j)) \one _{\Gamma_j}| \F_s\right] = \sum_{j=1}^n P_{s,t}\Psi(\eta_j)\one _{\Gamma_j}= P_{s,t}\Psi(\eta).$
	The proof is completed by an approximation argument, namely if $\EE[|\eta|^2]<\infty$, then we can find a sequence $\eta_n$ of simple functions such that $\eta_n\rightarrow \eta$ in $L^2_{\F_s}(V)$ and \eqref{eq:markov_prop} holds, but then, up to a subsequence, $\eta_n\rightarrow \eta$ $\PP$-a.s. in $V$ and thus $\Psi(\phi(t, s, \eta_n))\rightarrow \Psi(\phi(t, s, \eta))$ $\PP-a.s.$ and the dominated convergence theorem give the desired identity (one can even use the stronger \Cref{prop:cont-init-data} to show the convergence).    
\end{proof}

	\begin{corollary}
		The transition semigroup $P_t$ is homogeneous, namely $P_{s,t}=P_{0, t-s} =P_{t-s}$. 
	\end{corollary}
	\begin{proof}
		It follows immediately since the equation's coefficients are time independent, see \cite[Corollary 9.16]{daprato_stochastic_}
	\end{proof}

\begin{proposition}\label{tightness}
	Let $\phi_0 \in L^p_{\F_{0}}(\Omega, V)$ and $\mu(s; \phi_0)= \mathscr{L}(\phi(t; \phi_0))$, then the sequence of laws $\nu^T= \frac{1}{T}\int _0^T \mu(s; \phi_0)ds$ is tight on $V$. 
\end{proposition}
\begin{proof}
	Thanks to the It\^{o} formula in \Cref{ito-formula}, we have, for some constant $C>0$
	$$\EE\left[\int_0^T |A^{1/2} \phi(s)|^2ds\right] \le CT.$$
	Hence, by Markov's inequality 
	$$\PP\left( \frac{1}{T}\int_0^T |A^{1/2} \phi(s)|^2ds \ge R\right) \le \frac{C}{R},$$
	and since $D(A^{1/2})\hookrightarrow V$ compactly thanks to \Cref{ass1.2} in \Cref{ass}, we obtain the thesis.
\end{proof}
\begin{proof}[Proof of \Cref{thm:ergod}, (Existence and energy balance)]
	Thanks to \Cref{tightness}, we have that $R_t^* \delta_{x_0} $ as defined in \Cref{sec:invariant} is a tight sequence. Moreover thanks to \Cref{feller} we have that $P_t$ is Feller, thus we can apply \Cref{thm:KB} to obtain an invariant measure. Finally, the energy balance \eqref{V_balance_stat} follows from \Cref{ito-formula} and stationarity (see for instance \cite[Theorem 3.3]{albeverio_spde_2008} for a detailed argument). 
\end{proof}
\subsection{Uniqueness via coupling}
In this section we prove the uniqueness of the invariant measure applying the theory of \Cref{sec:coupling-method}. We begin proving an exponential estimate that we will need later. 
\begin{proposition}\label{prop:exp_est}
	Let $\phi$ be a solution of \eqref{eq:main} in the sense of \Cref{def-sol-H1}, then the following exponential bound holds
	\begin{align*}
		\PP\left( \sup_{t\ge 0} \left(\|\phi(t)\|_V^2 - \|\phi_0\|_V^2 + \frac{1}{2}\int_0^t \|\phi(s)\|_W^2 ds - \kappa t \right) \ge R\right) \le e^{-\frac{R}{2C\kappa}},
	\end{align*}
\end{proposition}
\begin{proof}
	The quadratic variation of $M_t= \brak{\phi, dW_t}$ satisfies 
	$$\brak{M}_t = \sum_{k}^{M}\int_0^t |\sigma_k|^2 \brak{f_k, \phi(s)}^2 ds \le \sum_{k}^{M}|\sigma_k|^2|(I-\Delta)^{-1/2}f_k|^2\int_0^t |(I-\Delta)^{1/2}\phi(s)|^2ds. $$
	Thus, up to a constant depending on the embedding $W\rightarrow V$ we get $\brak{M}_t \le C\kappa\int_0^t \|\phi(s)\|_W^2ds$ thus we have, from the It\^{o} formula \Cref{ito-formula}, for every $\gamma >0$ 
	\begin{align*}
		\|\phi (t)\|_V^2 -  \|\phi_0\|_V^2 + 2(1- {\gamma}C\kappa)\int_0^t\|\phi(s)\|_W^2 ds - \kappa t\le 2 \int_0^t\brak{\phi , dW_s}  -  \gamma \brak{M}_t.
	\end{align*}
	Choosing $\gamma = \frac{1}{2C\kappa}$ and setting $N_t =\int_0^t\brak{\phi , dW_s}  -   \frac{\gamma}{2}\brak{M}_t$ we get 
	\begin{align*}\PP&\left(  \sup_{t\ge 0}\left(\|\phi(t)\|_V^2- \|\phi_0\|_V^2 + \int_0^t\|\phi(s)\|_W^2ds - \kappa \right) \ge R \right)\\ & \le \PP\left(\sup_{t\ge 0} \left( \int_0^t\brak{\phi , dW_s}  -   \frac{\gamma}{2}\brak{M}_t \right)\ge \frac{R}{2}\right) \\
		& \le \PP\left(\sup_{t\ge 0} e^{\gamma N_t} \ge e^{\frac{\gamma R}{2}} \right) \\
		&\le e^{-\frac{ R}{4C{\kappa}}}.
	\end{align*}
	where in the last step we have used the maximal inequality \cite[Theorem 5.9]{baldi_stochastic_2017}, applied to the supermartingale $e^{\gamma N_t}$ (see also \cite[Sec 8.2]{baldi_stochastic_2017} for a proof that $e^{\gamma N_t}$ is a supermartingale).
\end{proof}
We are now ready to introduce the coupling argument: Recall from \Cref{sec:funct-sett} the definition of the finite-dimensional orthogonal projection $P_N$ and consider $ P_N^\Delta := (I-\Delta)P_N$.
For any $\phi_0, \psi_0 \in V$ consider the following system of stochastic equations. 
\begin{equation}\label{eq:forced}
	\begin{cases}
		d(\phi -\Delta \phi) + \left(A \phi + \NN(\phi, v-\Delta \phi)  \right)dt =  dW_t \\
		d(\psi -\Delta \psi) + \left(A \psi + \NN\left(\psi, v-\Delta \psi  \right)\right)dt + \lambda \one_{ \{t \le \tau_K  \}} P^\Delta_N(\psi - \phi )=  dW_t \\
		\phi(0)= \phi_0, \quad  \psi(0)= \psi_0,
	\end{cases}
\end{equation}
where
\begin{equation*}
	\tau_K:= \inf \left\{t>0: \ \int_0^t|P^\Delta_N(\psi - \phi)|^2  ds\ge K\right\}.
\end{equation*}
Pathwise unique global weak solutions (in the sense analogous to \Cref{def-sol-H1}) of this systems living in $C(0, T; V\times V) \cap L^2(0, T; W\times W)$ for every $T>0$ can be constructed by an easy adaptation of the arguments of \Cref{sec:WP}, with analogous energy estimates. In particular the processes $\phi$ and $\psi$ $\PP-$almost surely belong to $C([0, \infty); V)$ when we equip this space with the norm $\|f\|_{C([0, \infty); V)}:= \sup_{t\ge 0} (1+t)^{-1}\|f(t)\|_V$).
We remark that the noise appearing in both equation is the same. Moreover, under the assumption of \Cref{thm:ergod} on $N, M$ and thanks to the definition of $\tau_K$, which allows to easily verify the Novikov condition, Girsanov's theorem ensure that the laws of $\psi$ and $\phi$ on $C([0, \infty); V)$ are equivalent (we give more details below), so that the joint laws form an asymptotic equivalent coupling in the sense introduced in \Cref{sec:coupling-method}. 
Let us consider $m= \psi- \phi$. It satisfies 
\begin{equation}
	\partial_t (m - \Delta m) + Am + \NN( \psi, v-\Delta \psi) -\NN(\phi, v-\Delta \phi) = \lambda \one_{\{\tau_k \ge t\}}P_N^\Delta m, 
\end{equation}
where the equality is understood in weak sense analogously to \Cref{def-sol-H1}. We now prove the following fundamental lemma 
\begin{proposition}\label{prop:convergence-sol}
	For every initial values $\phi_0, \psi_0 \in V$, there exists $N= N(\kappa)$ large enough, $\lambda, K >0$ such that 
	$$\PP\left( \lim_{t\rightarrow+\infty}\left\|\phi(t, \phi_0) - \psi(t, \psi_0)\right\|_{V}=0\right) >0.$$
\end{proposition}
\begin{proof}
	Thanks to the regularities of $\phi, \psi$, analogously as observed in \Cref{proof_uniq},
	$$(I-\Delta)^{1/2}m \in L^\infty(0, T; D(A^{1/2}(I-\Delta)^{-1/2})),$$
	$$\partial_t (I-\Delta)^{1/2}m \in L^2(0, T; D(A^{1/2}(I-\Delta)^{-1/2})').$$
	Lions-Magenes lemma \cite[Chapter III, Lemma 1.2]{temam_navierstokes_}, yields that $m$ is almost everywhere equal to a continuous function into $V$ and (recalling in particular \eqref{eq:nonlinear_trick}) it holds
	$$\frac{d}{dt}\|m\|_V^2 + 2|A^{1/2} m |^2 + \brak{\NN(m, \Delta m), \phi} + 2\lambda \one _{\{\tau_K \ge t\}}\brak{P_N^\Delta m,m} =0.$$
	The same estimates as in \Cref{proof_uniq} give
	\begin{align*}
		\brak{\NN(m, \Delta m), \phi} &\le \frac{1}{4}  |A^{1/2}m| ^2 + C'|A^{1/2}\phi|^2 |\nabla m|^2,
	\end{align*}
	where $C'$ is an absolute constant independent of everything. On the other hand, choosing $\lambda= \lambda_N/2$ we have, thanks to \Cref{lem:proj-est2} 
	\begin{align*}
		|A^{1/2} m |^2 + \lambda_N \one _{\{\tau_K \ge t\}}\brak{P_N^\Delta m,m} \ge \lambda_N\one_{\left\{\tau_K \le t\right\}} \| m\|_V^2.
	\end{align*}
	Thus by Gronwall we obtain, for every $t\le \tau_K$ 
	$$\|m(t)\|_V^2 \le \|m_0\|_V^2  \exp\left( C'\int_0^t |A^{1/2}\phi(s)|^2 ds - t \lambda_N\right).$$
	Now consider the set 
	\begin{align*}
		E_R=\left\{ \omega\in \Omega \ : \  \sup_{t\ge 0} \left(\|\phi(t)\|_V^2 + \int_0^t |A^{1/2}\phi(s)|^2 ds -  2\kappa t - \|\phi_0\|_V^2\right) \le R\right\}.
	\end{align*}
	In view of \Cref{prop:exp_est}, for $R:= R(\|v\|_V^2, \kappa, |\phi_0|)$ large enough this set has positive probability, moreover, for $\omega \in E_R$ it holds 
	\begin{equation*}
		\int_0^t |A^{1/2}\phi(s)|^2 ds \le  R +  2\kappa t + \|\phi_0\|_V^2.
	\end{equation*}
	Thus, still on $E_R$ it holds, for every $t\le \tau_K$
	\begin{equation*}
		\|m(t)\|_V^2 \le e^{C'\left(R + \|\phi_0\|_V^2 \right)} e^{ (2C'\kappa-\lambda_N)t}.
	\end{equation*}
	By picking $N:=N(\kappa)>1$ large enough such that 
	$$2C'\kappa - \frac{\lambda_N}{2} \le 0,$$
	we obtain 
	$$\|m(t)\|_V^2 \le \|m(0)\|_V^2 \exp\left(C'( R + \|\phi_0\|_V^2)\right) e^{-\frac{\rho_N}{2} t}, \qquad t\le \tau_K.$$
	It follows that for $K:= K(R, \kappa, |\phi_0|^2)$ sufficiently large $E_R\subset \{\tau_K = \infty\}$ and on the set $E_R$ of positive probability, $m(t)\rightarrow 0 $ in $V$ as $t\rightarrow + \infty$.  
\end{proof}
We have all ingredients to prove ergodicity
\begin{proof}[Proof of \Cref{thm:ergod}]
	Fix $T>0$ and let $t_n = nT$. Then consider 
	$$\boldsymbol{\phi} = \left(\phi(t_1, \phi_0),  \phi(t_2; \phi_0), \ldots \right)\qquad \boldsymbol{\psi} = \left(\psi(t_1, \psi_0),  \psi(t_2; \psi_0), \ldots \right)$$
	and let $\boldsymbol{m}$, $\boldsymbol{n}$ be their laws on $V^\N$. In view of Girsanov's theorem \cite[Theorem 10.14]{daprato_stochastic_} and the definition of $\tau_K$, for every $\lambda, K\ge 0$, 
	$$W_t, \qquad W_t + \lambda \int_0^t\one_{s \le \tau_K}P_N^\Delta (\psi(s) - \phi(s))ds$$
	have equivalent laws on $C([0, \infty), V')$, provided that $(I-\Delta)P^{N}(V)\subset \operatorname{span}\left(\{f_1, \ldots f_M\}\right) $, which is exactly our assumption. As a consequence (\cite[Theorem 10.18]{daprato_stochastic_}), $\psi$ and $\phi$ have equivalent laws on $C([0, \infty); V)$. Considering now the law $\Gamma$ on $V^\N \times V^{\N}$ given by the couple $(\boldsymbol{\phi}, \boldsymbol{\psi})$ it follows that $\Gamma \in \tilde C(\delta_{\phi_0}P^{\N}, \delta_{\psi_0}P^{\N})$. Recalling the definitions in \Cref{sec:coupling-method}, we see that for $\rho(x,y):= \|(x-y)\|_V$,  $\G_\rho$ determines measure on $(V, \rho)$ and $D_\rho$ is a measurable set of $V^\N \times V^\N$. Moreover for each $\phi_0, \psi_0 \in V$, thanks to \Cref{prop:convergence-sol} we have proved that there exists $\lambda, K>0$ and $N=N(\kappa)$ large enough such that $\phi(t; \phi_0)- \psi(t, \psi_0)\rightarrow0$ in $V$ for $t\rightarrow + \infty$ on a set of positive probability, namely $\Gamma(D_\rho)>0$. Thanks again to the assumption that $(I-\Delta)P_N(V)\subset \operatorname{span}\left(\{f_1, \ldots f_M\}\right) $, we conclude that $\Gamma$ is an asymptotic equivalent coupling. It follows then from \Cref{ergodicity_cor} that there exists at most one ergodic invariant measure $\mu$ for \eqref{eq:main} on $V$.\\
\end{proof}

\begin{acknowledgements}
The research of F.F. and C.C. is funded
by the European Union (ERC, NoisyFluid, No. 101053472). Views and opinions expressed are however those of the authors only and do not necessarily reflect those of the European Union or the European Research Council. Neither the European Union nor the granting authority can be held responsible for them.
\end{acknowledgements}
\bibliographystyle{plain}
\bibliography{references, references-2}

\end{document}